\documentclass[a4paper, 1wpt]{amsart}
\usepackage[T2A]{fontenc}
\usepackage[english]{babel}
\usepackage{graphics}
\usepackage{amsfonts, amssymb, amscd, amsmath}
\usepackage[dvips]{graphicx}
\usepackage{latexsym}
\usepackage[matrix,arrow,curve]{xy}
\usepackage{mathabx,mathtools}
\usepackage{scalerel}
\usepackage{color}
\usepackage{tikz}
\usetikzlibrary{shapes}
\usetikzlibrary{positioning,arrows}
\usepackage[utf8x]{inputenc}
\usepackage{scalefnt}

\DeclareMathOperator{\link}{link}

\DeclareMathOperator{\rk}{rank}

 \DeclareMathOperator{\sta}{star}
\DeclareMathOperator{\ft}{ft} \DeclareMathOperator{\wid}{wid}
\DeclareMathOperator{\ver}{Vert}

\DeclareMathOperator*{\bighash}{\scalerel*{\hash}{\textstyle\sum}}

\newcommand{\Zo}{\mathbb{Z}}
\newcommand{\Ro}{\mathbb{R}}
\newcommand{\Co}{\mathbb{C}}

\newcommand{\Zt}{\Zo_2}

\newcommand{\eqd}{\stackrel{\text{\tiny def}}{=}}

\newcommand{\congeq}{\stackrel{T}{\cong}}

\newcommand{\csum}[2]{\mathbin{_{#1}\!\hash_{#2}}}

\newcommand{\Sing}{Z}
\newcommand{\ta}{\mathfrak{t}}
\newcommand{\PsiE}{\Psi_E}
\newcommand{\PsiV}{\Psi_V}

\newcommand{\F}{\mathcal{F}}
\newcommand{\Dd}{\mathcal{D}}
\newcommand{\Ff}{\mathcal{F}}
\newcommand{\Ss}{\mathbb{S}}

\newcommand{\CP}{\mathbb{C}P}

\newcommand{\temp}{(\Gamma,\Psi_V,\Psi_E)}
\newcommand{\staro}{\sta^{\circ}}
\newcommand{\Ot}{O}
\newcommand{\Pp}{\mathcal{P}}
\newcommand{\bsl}{\scalerel*{/}{\sum}}
\newcommand{\bsm}{\scalerel*{\setminus}{\sum}}

\newcounter{stmcounter}[section]
\newcounter{thcounter}

\numberwithin{equation}{section}

\theoremstyle{plain}

\newtheorem{thm}[thcounter]{Theorem}
\newtheorem{prop}[stmcounter]{Proposition}
\newtheorem{lemma}[stmcounter]{Lemma}
\newtheorem{defin}[stmcounter]{Definition}
\newtheorem{problem}{Problem}

\newtheorem{claim}[stmcounter]{Claim}

\theoremstyle{definition}

\newtheorem{ex}[stmcounter]{Example}
\newtheorem{rem}[stmcounter]{Remark}
\newtheorem{con}[stmcounter]{Construction}

\begin{document}

\title{Toric origami structures on quasitoric manifolds}

\author[A. Ayzenberg]{Anton Ayzenberg}
\address{Department of Mathematics, Osaka City University, Sumiyoshi-ku, Osaka 558-8585, Japan.}
\email{ayzenberga@gmail.com}

\author[M. Masuda]{Mikiya Masuda}
\address{Department of Mathematics, Osaka City University, Sumiyoshi-ku, Osaka 558-8585, Japan.}
\email{masuda@sci.osaka-cu.ac.jp}

\author[S. Park]{Seonjeong Park}
\address{Division of Mathematical Models, National Institute for Mathematical Sciences, 463-1 Jeonmin-dong, Yuseong-gu, Daejeon 305-811, Korea}
\email{seonjeong1124@nims.re.kr}

\author[H. Zeng]{Haozhi Zeng}
\address{Department of Mathematics, Osaka City University, Sumiyoshi-ku, Osaka 558-8585, Japan.}
\email{zenghaozhi@icloud.com}


\date{\today}
\thanks{The first author is supported by the JSPS postdoctoral fellowship program.
The second author was partially supported by Grant-in-Aid for
Scientific Research 25400095}



\begin{abstract}
We construct quasitoric manifolds of dimension 6 and higher which
are not equivariantly homeomorphic to any toric origami manifold.
All necessary topological definitions and combinatorial
constructions are given and the statement is reformulated in
discrete geometrical terms. The problem reduces to existence of
planar triangulations with certain coloring and metric properties.
\end{abstract}

\maketitle

\section*{Introduction}

Origami manifolds appeared in differential geometry recently as a
generalization of symplectic manifolds \cite{ca-gu-pi}. Toric
origami manifolds are in turn generalizations of symplectic toric
manifolds. Toric origami manifolds are a special class of
$2n$-dimensional compact manifolds with an effective action of a
half-dimensional compact torus $T^n$. In this paper we consider
the following question. How large is this class? Which manifolds
with half-dimensional torus actions are toric origami manifolds?

Since the notion of a manifold with an effective half-dimensional
torus action is too general to deal with, we restrict to
quasitoric manifolds. This class of manifolds is large enough to
include many interesting examples, and small enough to keep
statements feasible. In \cite{ma-pa} Masuda and Park proved

\begin{thm}\label{thm4dim}
Any simply connected compact smooth $4$-manifold $M$ with an
effective smooth action of $T^2$ is equivariantly diffeomorphic to
a toric origami manifold.
\end{thm}

In particular, any $4$-dimensional quasitoric manifold is toric
origami. The same question about higher dimensions was open. In
this paper we give the negative answer.

\begin{thm}\label{thmQtorNotOrigami}
For any $n\geqslant 3$ there exist $2n$-dimensional quasitoric
manifolds, which are not equivariantly homeomorphic to any toric
origami manifold.
\end{thm}

We will describe an obstruction for a quasitoric $6$-manifold to
be toric origami and present a large series of examples, where
such an obstruction appears. Existence of such examples in higher
dimensions follows from $6$-dimensional case. In spite of
topological nature of the task, the proof is purely discrete
geometrical: it relies on metric and coloring properties of planar
graphs. Thus we tried to separate the discussion of established
facts in toric topology which motivated this study, from the proof
of the main theorem to keep things comprehensible for the broad
audience.

The paper is organized as follows. In section~\ref{sectTopology}
we briefly review the necessary topological objects, and describe
the standard combinatorial and geometrical models which are used
to classify them. The objects are: quasitoric manifolds,
symplectic toric manifolds, and toric origami manifolds. The
corresponding combinatorial models are: characteristic pairs,
Delzant polytopes, and origami templates respectively. In
section~\ref{sectWeightedSpheres} we introduce the notion of a
weighted simplicial cell sphere, which, in a certain sense,
unifies all these combinatorial models. We define a connected sum
of weighted spheres along vertices. This operation is dual to the
operation of producing an origami template from Delzant polytopes.
It plays an important role in the proof. Section~\ref{sectProof}
contains the combinatorial statement from which follows
Theorem~\ref{thmQtorNotOrigami}, and the proof of this statement.
The interaction of our study with the study of the Brownian map
allows to prove that asymptotically most simple 3-polytopes admit
quasitoric manifolds which are not toric origami. We describe this
interaction as well as other adjacent questions in the last
section~\ref{sectConclusion}.

Authors are grateful to the anonymous referee for his comments on
the previous version of the paper.

%
%
%
%
%
%
%

\section{Topological preliminaries}\label{sectTopology}

%
%

\subsection{Quasitoric manifolds}

The subject of this subsection originally appeared in the seminal
work of Davis and Januszkiewicz \cite{da-ja}. The modern
exposition and technical details can be found in
\cite[Ch.7]{bu-pa14}.

Let $T^n$ be a compact $n$-dimensional torus. The standard
representation of $T^n$ is a representation
$T^n\curvearrowright\Co^n$ by coordinate-wise rotations, i.e.
\[(t_1,\ldots,t_n)\cdot(z_1,\ldots,z_n)=(t_1z_1,\ldots,t_nz_n),\]
for $z_i,t_i\in \Co$, $|t_i|=1$.

The action of $T^n$ on a smooth manifold $M^{2n}$ is called
locally standard, if $M$ has an atlas of standard charts, each
isomorphic to a subset of the standard representation. More
precisely, a standard chart on $M$ is a triple $(U,f,\psi)$, where
$U\subset M$ is a $T^n$-invariant open subset, $\psi$ is an
automorphism of $T^n$, and $f$ is a $\psi$-equivariant
homeomorphism $f\colon U\to W$ onto a $T^n$-invariant open subset
$W\subset\Co^n$ (i.e. $f(t\cdot y) = \psi(t)\cdot f(y)$ for all
$t\in T^n$, $y\in U$). In the following $M$ is supposed to be
compact.

Since the orbit space $\Co^n/T^n$ of the standard representation
is a nonnegative cone $\Ro_{\geqslant}^n=\{x\in \Ro^n\mid
x_i\geqslant 0\}$, the orbit space of any locally standard action
obtains the structure of a compact smooth manifold with corners.
Recall that a manifold with corners is a topological space locally
modeled by open subsets of $\Ro_{\geqslant}^n$ with the
combinatorial stratification induced from the face structure of
$\Ro_{\geqslant}^n$. There are many technical details about the
formal definition which we left beyond the scope of this work (the
exposition relevant to our study can be found in \cite{bu-pa14} or
\cite{yo}).

The orbit space $Q=M/T^n$ of a locally standard action carries an
additional information about stabilizers of the action, called a
characteristic function. Let $\F(Q)$ denote the set of facets of
$Q$ (i.e. faces of codimension $1$). For each facet $F$ of $Q$
consider a stabilizer subgroup $\lambda(F)\subset T^n$ of points
in the interior of $F$. This subgroup is $1$-dimensional and
connected, thus it has the form
$\{(t^{\lambda_1},\ldots,t^{\lambda_n})\mid t\in T^1\}\subset
T^n$, for some primitive integral vector
$(\lambda_1,\ldots,\lambda_n)\in \Zo^n$, defined uniquely up to a
common sign. Thus, a primitive integral vector (up to sign)
$\Lambda(F)\in \Zo^n/\pm$ is associated with any facet $F$ of $Q$.
This map $\Lambda\colon \F(Q)\to \Zo^n/\pm$ is called a
\emph{characteristic function} (or a characteristic map). It
satisfies the following so called \eqref{eqStarCond}-condition:

\begin{equation}\label{eqStarCond}
\parbox[c]{10cm}{\center{
If facets $F_1,\ldots,F_s$ intersect, then the vectors
$\Lambda(F_1),\ldots,\Lambda(F_s)$ span a direct summand of
$\Zo^n$.}}\tag{$*$}
\end{equation}
Here we actually take not a class $\Lambda(F_i)\in \Zo^n/\pm$, but
one of its two particular representatives in $\Zo^n$. Obviously,
the condition does not depend on the choice of sign, thus
\eqref{eqStarCond} is well defined. The same convention appears
further in the text without special mention.

It is convenient to view the characteristic function $\Lambda$ on
$Q=M/T^n$ as a generalized coloring of facets. We assign primitive
integral vectors to facets instead of simple colors, and condition
\eqref{eqStarCond} is the requirement for this ``coloring'' to be
proper. The general idea, which simplifies many considerations in
toric topology is that the combinatorial structure of the orbit
space $Q$ together with the assigned coloring completely encodes
the equivariant homeomorphism type of $M$ in many cases. The
precise statement also involves the so called Euler class of the
action, which is an element of $H^2(Q;\Zo^n)$, and allows to
classify all compact smooth manifolds with locally standard torus
actions. The reader may find this general statement in \cite{yo}.

Anyway, we will work with only a special type of locally standard
torus actions, namely quasitoric manifolds. In this special case
Euler class vanishes, so we will not care about it.

\begin{defin}\label{definQtoric}
A manifold $M^{2n}$ with a locally standard action of $T^n$ is
called quasitoric, if the orbit space $M/T^n$ is homeomorphic to a
simple polytope as a manifold with corners.
\end{defin}

Recall that a convex polytope $P$ of dimension $n$ is called
simple if any of its vertices lies in exactly $n$ facets. In other
words, a simple polytope is a polytope which is at the same time a
manifold with corners. Considering manifolds with corners, simple
polytopes are the simplest geometrical examples one can imagine.
This makes the definition of quasitoric manifold very natural.

Let $P$ be a simple polytope and $\Lambda$ be a characteristic
function, i.e. any map $\Lambda\colon\F(P)\to \Zo^n/\pm$
satisfying \eqref{eqStarCond}-condition. The pair $(P,\Lambda)$ is
called a \emph{characteristic pair}. According to \cite{da-ja},
there is a one-to-one correspondence
\[
\{\mbox{quasitoric
manifolds}\}\leftrightsquigarrow\{\mbox{characteristic pairs}\}
\]
up to equivariant homeomorphism on the left-hand side and
combinatorial equivalence on the right-hand side. Given a
characteristic pair, one can construct the corresponding
quasitoric manifold explicitly.

\begin{con}[Model of quasitoric manifold]\label{conModelQtoric}
Let $(P,\Lambda)$ be a characteristic pair, $\dim P=n$. Consider a
topological space
\begin{equation}\label{eqQtoricModel}
M_{(P,\Lambda)}=P\times T^n / \sim.
\end{equation}
The equivalence $\sim$ is generated by relations $(p,t_1)\sim
(p,t_2)$ where $p$ lies in a facet $F\in \F(P)$ and
$t_1t_2^{-1}\in \lambda(F)$. The torus $T^n$ acts on
$M_{(P,\Lambda)}$ by rotating second coordinate and the orbit
space $M_{(P,\Lambda)}/T^n$ is isomorphic to $P$. Condition
\eqref{eqStarCond} implies that the action is locally standard.
There is a smooth structure on $M_{(P,\Lambda)}$ and the action of
$T^n$ is smooth (the construction of smooth structure can be found
in \cite{bu-pa07}). Therefore, $M_{(P,\Lambda)}$ is a quasitoric
manifold.
\end{con}

Let $\eta$ denote the projection to the orbit space $\eta\colon
M_{(P,\Lambda)}\to P$. Each facet $F\in \F(P)$ determines a smooth
submanifold $N_F\eqd \eta^{-1}(F)\subset M_{(P,\Lambda)}$ of
dimension $2n-2$, called characteristic submanifold. On its own,
the manifold $N_F$ is again a quasitoric manifold with the orbit
space $F$.

%
%

\subsection{Toric origami manifolds}

In the following subsections we recall the definitions and
properties of toric origami manifolds and origami templates. More
detailed exposition of this theory can be found in
\cite{ca-gu-pi}, \cite{ma-pa} or \cite{ho-pi12}.

A \emph{folded symplectic form} on a $2n$-dimensional smooth
manifold $M$ is a closed $2$-form $\omega$ such that

\begin{itemize}
\item Its top power $\omega^n$ is transversall to the zero section of
$\Lambda^{2n}(T^*M)$. As a consequence $\omega^n$ vanishes on a
smooth submanifold $\Sing\subset M$ of codimension 1.
\item The restriction of $\omega$ to $\Sing$ has maximal rank.
\end{itemize}

The hypersurface $\Sing$ where $\omega$ is degenerate is called
the \emph{fold}. The pair $(M, \omega)$ is called a \emph{folded
symplectic manifold}. If $\Sing$ is empty, $\omega$ is a genuine
symplectic form and $(M, \omega)$ is a genuine symplectic manifold
according to classical definition.

The reader may get a feeling of this notion by working locally.
Darboux's theorem says that any symplectic form can be written
locally as $\sum_i dx_i\wedge dy_i$ in appropriate coordinates.
The folded forms are exactly the forms written as
\[x_1dx_1\wedge dy_1 + \sum_{i>1} dx_i\wedge dy_i\]
in appropriate coordinates (for this analogue of Darboux's theorem
see \cite{ca-gu-pi} and references therein). The fold $\Sing$ is
thus a hypersurface given locally by $x_1=0$.

Since the restriction of $\omega$ to $\Sing$ has maximal rank, it
has a one-dimensional kernel at each point of $\Sing$. This
determines a line field on $\Sing$ called the \emph{null
foliation}. If the null foliation is the vertical bundle of some
principal $S^1$-fibration $\Sing\to Y$ over a compact base $Y$,
then the folded symplectic form $\omega$ is called an
\emph{origami form} and the pair $(M,\omega)$ is called an
\emph{origami manifold}.

The action of a torus $T$ (of any dimension) on an origami
manifold $(M,\omega)$ is called Hamiltonian if it admits a moment
map $\mu\colon M\to \ta^*$ to the dual Lie algebra of the torus,
which satisfies the conditions: (1) $\mu$ is equivariant with
respect to the given action of $T$ on $M$ and the coadjoint action
of $T$ on the vector space $\ta^*$ (by commutativity of torus this
action is trivial); (2) $\mu$ collects Hamiltonian functions, that
is, $d\langle\mu,V\rangle = \omega(\tilde{V},\cdot)$, where
$\langle\mu,V\rangle$ is the function on $M$, taking the value
$\langle\mu(x),V\rangle$ at a point $x\in M$, $\tilde{V}$ is a
vector flow on $M$, generated by $V\in\ta$, and
$\omega(\tilde{V},\cdot)$ is its dual 1-form.

\begin{defin}
A toric origami manifold $(M,\omega,T,\mu)$, abbreviated as $M$,
is a compact connected origami manifold $(M,\omega)$ equipped with
an effective Hamiltonian action of a torus $T$ with $\dim T =
\frac12\dim M$ and with a choice of a corresponding moment map
$\mu$.
\end{defin}

%
%

\subsection{Symplectic toric manifolds}

When the fold $\Sing$ is empty, a toric origami manifold is a
\emph{symplectic toric manifold}. In this case the image $\mu(M)$
of the moment map is a Delzant polytope in $\ta^*$, and the map
$\mu\colon M\to \mu(M)$ itself can be identified with the map to
the orbit space $\eta\colon M\to M/T^n$. A classical theorem of
Delzant \cite{del} says that symplectic toric manifolds are
classified by the images of their moment maps in $\ta^*\cong
\Ro^n$. In other words, there is a one-to-one correspondence
\[
\{\mbox{symplectic toric
manifolds}\}\leftrightsquigarrow\{\mbox{Delzant polytopes}\}
\]
up to equivariant symplectomorphism on the left-hand side, and
affine equivalence on the right-hand side. Let us recall the
notion of Delzant polytope.

\begin{defin}
A simple convex polytope $P\subset \Ro^n$ is called Delzant, if
its normal fan is smooth (with respect to a given lattice
$\Zo^n\subset\Ro^n$). In other words, all normal vectors to facets
of $P$ have rational coordinates, and, whenever facets
$F_1,\ldots,F_n$ meet in a vertex of $P$, the primitive normal
vectors $\nu(F_1),\ldots,\nu(F_n)$ form a basis of the lattice
$\Zo^n$.
\end{defin}

Let $M_P$ be the symplectic toric manifold corresponding to
Delzant polytope $P$. We do not need the construction of Delzant
correspondence in full generality, but we need to review the
topological construction of symplectic toric manifold
corresponding to a given Delzant polytope. Forgetting the
symplectic structure, any symplectic toric manifold, as a smooth
manifold with $T^n$-action, is a quasitoric manifold.

\begin{con}[Topological model of symplectic toric
manifold]\label{conModelSymplectic} Let $P$ be a Delzant polytope
in $\Ro^n$. For a facet $F\in \F(P)$ consider its outward
primitive normal vector $\tilde{\nu}(F)\in \Zo^n$. Consider the
corresponding vector modulo sign: $\nu(F)\in\Zo^n/\pm$. By the
definition of Delzant polytope, $\nu\colon\F(P)\to \Zo^n/\pm$
satisfies \eqref{eqStarCond}, thus provides an example of a
characteristic function. The manifold $M_P$ is equivariantly
diffeomorphic to quasitoric manifold $M_{(P,\Lambda)}$
corresponding to the characteristic pair $(P,\Lambda)$.
\end{con}

%
%

\subsection{Origami templates}

Delzant theorem provides a one-to-one correspondence between
symplectic toric manifolds and Delzant polytopes. To generalize
this correspondence to toric origami manifolds we need a notion of
an origami template, which we review next.

Let $\Dd_n$ denote the set of all (full-dimensional) Delzant
polytopes in $\Ro^n$ (w.r.t. a given lattice) and $\Ff_n$ the set
of all their facets.

\begin{defin}\label{definOrigamiTemplate}
An origami template is a triple $(\Gamma,\PsiV,\PsiE)$, where
\begin{itemize}
\item $\Gamma$ is a connected finite graph (loops and multiple edges
are allowed) with the vertex set $V$ and edge set $E$;
\item $\Psi_V$ is a map, which associates to any vertex of $\Gamma$ a
full-dimensional Delzant polytope, $\PsiV\colon V\to \Dd_n$;
\item $\Psi_E$ is a map, which associates to any edge of $\Gamma$
a facet of Delzant polytope, $\PsiE\colon E\to \Ff_n$;
\end{itemize}
subject to the following conditions:
\begin{enumerate}
\item[1.] If $e\in E$ is an edge of $\Gamma$ with endpoints $v_1, v_2\in
V$, then $\PsiE(e)$ is a facet of both polytopes $\PsiV(v_1)$ and
$\PsiV(v_2)$, and these polytopes coincide near $\PsiE(e)$ (this
means there exists an open neighborhood $U$ of $\PsiE(e)$ in
$\Ro^n$ such that $U\cap\PsiV(v_1)=U\cap\PsiV(v_2)$).
\item[2.] If $e_1,e_2\in E$ are two edges of $\Gamma$ adjacent to $v\in V$,
then $\PsiE(e_1)$ and $\PsiE(e_2)$ are disjoint facets of
$\PsiV(v)$.
\end{enumerate}
The facets of the form $\PsiE(e)$ for $e\in E$ are called the fold
facets of the origami template.
\end{defin}

For convenience in the following we call the vertices of graph
$\Gamma$ the nodes.

One can simply view an origami template as a collection of
(possibly overlapping) Delzant polytopes $\{\PsiV(v)\mid v\in V\}$
in the same ambient space, with some gluing data, encoded by a
template graph $\Gamma$. When $n=2$, the picture looks like a
folded sheet of paper on a flat plane, which is one of the
explanations for the term ``origami'' (see
Fig.~\ref{fig:templates}). Nevertheless, to avoid the confusion,
we should mention that most flat origami models in a common sense
are not origami templates in the sense of
Definition~\ref{definOrigamiTemplate}.

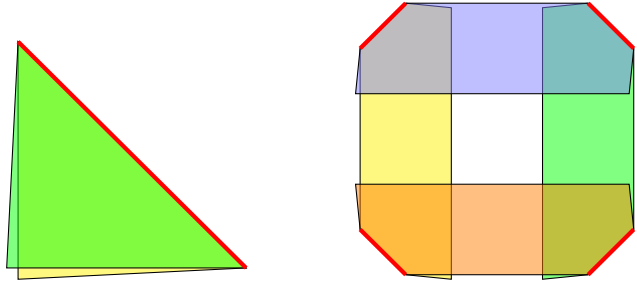
\begin{figure}[h]
    \begin{center}
        \begin{tikzpicture}[scale=3]
            \pgfsetfillopacity{0.5}
            \filldraw[fill=yellow] (0,-0.05)--(1,0)--(0,1)--cycle;
            \filldraw[fill=green] (-0.05,0)--(1,0)--(0,1)--cycle;
            \draw[ultra thick,red] (1,0)--(0,1);
        \end{tikzpicture}\qquad\qquad
        \begin{tikzpicture}[scale=.6]
            \pgfsetfillopacity{0.5}
            \filldraw[fill=yellow] (1,0)--(2,-0.1)--(2,5.9)--(1,6)--(0,5)--(0,1)--cycle;
            \filldraw[fill=green] (4,-0.1)--(5,0)--(6,1)--(6,5)--(5,6)--(4,5.9)--cycle;
            \filldraw[fill=orange] (-0.1,2)--(0,1)--(1,0)--(5,0)--(6,1)--(5.9,2)--cycle;
            \filldraw[fill=blue!50] (-0.1,4)--(5.9,4)--(6,5)--(5,6)--(1,6)--(0,5)--cycle;
            \draw[ultra thick,red] (1,0)--(0,1);
            \draw[ultra thick,red] (5,0)--(6,1);
            \draw[ultra thick,red] (6,5)--(5,6);
            \draw[ultra thick,red] (0,5)--(1,6);
        \end{tikzpicture}
    \end{center}
    \caption{Examples of origami templates in $\dim = 2$. Fold facets are shown in red.
    The lines which should actually coincide are drawn close to each other for convenience.}
\label{fig:templates}
\end{figure}

The following is a generalization of Delzant's theorem to toric
origami manifolds.

\begin{thm}[\cite{ca-gu-pi}]\label{theo:classifOrigami}
Assigning the moment data of a toric origami manifold induces a
one-to-one correspondence
\[
\{\mbox{toric origami
manifolds}\}\leftrightsquigarrow\{\mbox{origami templates}\}
\]
up to equivariant origami symplectomorphism on the left-hand side,
and affine equivalence on the right-hand side.
\end{thm}

Let $M_O=M_{(\Gamma,\Psi_V,\Psi_E)}$ be the toric origami manifold
corresponding to origami template $O=(\Gamma,\Psi_V,\Psi_E)$. As
in the case of symplectic toric manifolds, we do not need the
construction of this correspondence in full generality. But we
give a topological construction of the toric origami manifold from
a given origami template.

\begin{con}[Topological model of toric origami manifold]\label{conModelOrigami}
Consider an origami template $\Ot=(\Gamma,\PsiV,\PsiE)$,
$\Gamma=(V,E)$. For each node $v\in V$ the Delzant polytope
$\PsiV(v)\in \Pp_n$ gives rise to a symplectic toric manifold
$M_{\PsiV(v)}$, see construction~\ref{conModelSymplectic}. Now do
the following procedure:
\begin{enumerate}
\item[1] Take a disjoint union of all manifolds $M_{\PsiV(v)}$ for $v\in
V$;

\item[2] For each edge $e\in E$ with distinct endpoints $v_1$ and $v_2$
take an equivariant connected sum of $M_{\PsiV(v_1)}$ and
$M_{\PsiV(v_2)}$ along the characteristic submanifold
$N_{\PsiE(e)}$ (which is embedded in both manifolds by
construction~\ref{conModelQtoric});

\item[3] For each loop $e\in E$ based at $v\in V$ take a real blow
up of normal bundle to the submanifold $N_{\PsiE(e)}$ inside
$M_{\PsiV(v)}$.
\end{enumerate}

Step 2 makes sense because of pt.1 of
Definition~\ref{definOrigamiTemplate}. Indeed, the polytopes
$\PsiV(v_1)$ and $\PsiV(v_2)$ agree near $\PsiE(e)$, thus
$M_{\PsiV(v_1)}$ and $M_{\PsiV(v_2)}$ have equivariantly
homeomorphic neighborhoods around $N_{\PsiE(e)}$, so the connected
sum is well defined. Pt. 2 of
Definition~\ref{definOrigamiTemplate} ensures that surgeries do
not touch each other, so all the connected sums and blow ups can
be taken simultaneously. The smooth manifold obtained by the above
procedure is the origami manifold $M_O$ up to equivariant
diffeomorphism.
\end{con}

\begin{rem}
By definition, the operation of equivariant connected sum consists
in cutting small equal $T^n$-invariant tubular neighborhoods of
$N_{\Psi_E(e)}$ in $M_{\Psi_V(v_1)}$ and $M_{\Psi_V(v_2)}$, and
then gluing the resulting manifolds by identity isomorphism of the
boundaries. The image of the moment map under this operation
becomes smaller. Thus the construction described above is
certainly not enough to prove the classificational theorem. In the
theorem one should not only take a connected sum but also attach
collars of the form $\Sing\times(-\varepsilon,\varepsilon)$ (see
details in \cite{ca-gu-pi}). Nevertheless, both constructions,
with collars and without collars, lead to the same result, up to
equivariant diffeomorphism.
\end{rem}

\begin{ex}\label{exSfour}
Let us construct a toric origami manifold $X$, corresponding to
the origami template, made of two triangles
(Fig.~\ref{fig:templates}, left). The symplectic toric 4-manifold
corresponding to a triangle is known to be the complex projective
plane $\CP^2$. The characteristic submanifold corresponding to the
fold facet is a projective line $\CP^1\subset \CP^2$. Thus, $X$ is
a connected sum of two copies of $\CP^2$ along the line $\CP^1$,
which lies in both. This has a simple geometrical interpretation.
If we consider $\CP^1\subset \CP^2$ as a projective line at
infinity, and denote the tubular neighborhood of this line by
$U(\CP^1)$, then $\CP^2\setminus U(\CP^1)$ is a 4-disk $D^4$. Thus
$X$ is a result of gluing two copies of $D^4$ by the identity
diffeomorphism of the boundary. Thus $X\cong S^4$. The action of
$T^2$ on $S^4$ is also easily described. Consider the space
$\Co^2\times \Ro$, and let $T^2$ act on $\Co^2$ by coordinate-wise
rotations, and trivially on $\Ro$. The unit sphere $S^4\subset
\Co^2\times \Ro$ is invariant under this action. This gives a
required action of $T^2$ on $X\cong S^4$.
\end{ex}

An origami template $\Ot=\temp$ is called orientable if the
template graph $\Gamma$ is bipartite, or, equivalently,
$2$-colorable. It is not hard to prove that the origami template
$\Ot$ is orientable whenever $M_{\Ot}$ is an orientable
manifold~\cite{ampz}.

An origami template $\Ot=\temp$ (and the corresponding manifold
$M_{\Ot}$) is called co\"{o}rientable if $\Gamma$ has no loops
(i.e. edges based at one point). Any orientable template (resp.
toric origami manifold) is co\"{o}rientable, because a graph with
loops is not $2$-colorable. If $M_{\Ot}$ is co\"{o}rientable, then
the action of $T^n$ on $M_{\Ot}$ is locally standard~\cite[lemma
5.1]{ho-pi12}. The converse is also true. If the template graph
has a loop, then the real normal blow up in Step 3 of
construction~\ref{conModelOrigami} implies existence of
$\Zt$-components in stabilizer subgroups. Therefore
non-co\"{o}rientable toric origami manifolds are not locally
standard. In the following we consider only co\"{o}rientable
templates and toric origami manifolds.

\begin{con}[Orbit space of toric origami manifold]\label{conOrbitOfOrigami}
The orbit space $Q=M_{\temp}/T^n$ of a (co\"{o}rientable) toric
origami manifold is a smooth manifold with corners. Its
homeomorphism type can be described as a topological space
obtained by gluing polytopes $\PsiV(v)$ along fold facets. More
precisely,
\begin{equation}\label{eqOrbitOfOrigami}
Q = \bigsqcup_{v\in V} (v,\PsiV(v))\bsl\sim,
\end{equation}
where $(u,x)\sim (v,y)$ if there exists an edge $e$ with endpoints
$u$ and $v$, and $x=y\in \PsiE(e)$. Facets of $Q$ are given by
non-fold facets of polytopes $\PsiV(v)$ identified in the same
way. To make this precise, let us call non-fold facets $F_1\in
\F(\PsiV(v_1))$ and $F_2\in\F(\PsiV(v_2))$ elementary neighboring
w.r.t. to the edge $e\in E$ (with endpoints $v_1$ and $v_2$) if
$F_1\cap\PsiE(e)=F_2\cap\PsiE(e)$. The relation of elementary
neighborliness generates an equivalence relation $\leftrightarrow$
on the set of all non-fold facets of all polytopes $\PsiV(v)$.
Define the facet $[F]$ of the orbit space $Q$ as a union of facets
in one equivalence class:
\begin{equation}\label{eqFacetOfOrigami}
[F] \quad\eqd\quad \bigsqcup_{\mathclap{\substack{v\in V, G\in\F(\PsiV(v)),\\
G \text{ is not fold}, G\leftrightarrow F}}}\,\, (v,G)
\bsl\sim,\qquad [F] \in \F(Q),
\end{equation}
where $\sim$ is the same as in \eqref{eqOrbitOfOrigami}.

Let us define a primitive normal vector to the facet $[F]$ of $Q$
by $\nu([F])\eqd \nu(F)\in\Zo^n/\pm$. It is well defined since
$\nu(F)=\nu(G)$ for $F\leftrightarrow G$.

Note that the relation of elementary neighborliness determines a
connected subgraph $\Gamma_{[F]}$ of $\Gamma$. All facets
$G\leftrightarrow F$ are Delzant and lie in the same hyperplane
$H_{[F]}$. Thus we obtain an induced origami template
\begin{equation}\label{eqInducedTemplate}
\Ot_{[F]}=(\Gamma_{[F]},\PsiV|_{\Gamma_{[F]}}\cap
H_{[F]},\PsiE|_{\Gamma_{[F]}}\cap H_{[F]})
\end{equation}
of dimension $n-1$. In particular, if $\eta\colon M_{\Ot}\to Q$
denotes the projection to the orbit space, then the characteristic
submanifold $\eta^{-1}([F])$ is the toric origami manifold of
dimension $2n-2$ generated by the origami template $\Ot_{[F]}$.

We had defined the facets of the orbit space $Q=M_{\Ot}/T^n$. All
other faces are defined as connected components of nonempty
intersections of facets. On the other hand, faces can be defined
similarly to facets --- by gluing faces of polytopes $\PsiV(v)$
which are neighborly in the same sense as before.
\end{con}

Extending the origami analogy, we can think of the orbit space $Q$
as ``unfolding'' the origami template and then forgetting the
angles adjacent to the former fold facets (remember that we have
to identify neighboring faces!).

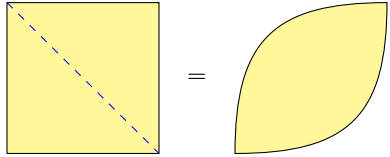
\begin{figure}[h]
\begin{center}
    \begin{tikzpicture}[scale=1]
        \filldraw[fill=yellow!50](0,0)--(2,0)--(2,2)--(0,2)--cycle;
        \draw[dashed, blue] (0,2)--(2,0);
        \node at (2.5,1) {$=$};
        \filldraw[fill=yellow!50] (3,0)..controls (4.5,0) and
        (5,0.5)..(5,2)..controls (3.5,2) and
        (3,1.5)..(3,0)--cycle;
    \end{tikzpicture}
\end{center}
\caption{The orbit space of a manifold $S^4$, corresponding to the
origami template shown on Fig.~\ref{fig:templates}, left.}
\label{fig:orbitspace}
\end{figure}

It is easy to see that the orbit space $Q=M_{\temp}/T^n$ has the
same homotopy type as the graph $\Gamma$, thus $Q$ is either
contractible (when $\Gamma$ is a tree) or homotopy equivalent to a
wedge of circles. This observation shows that whenever the
template graph $\Gamma$ has cycles, the corresponding toric
origami manifold cannot be quasitoric (recall that the orbit space
of quasitoric manifold is a polytope, which is contractible). As
an example, the origami template shown on
Fig.~\ref{fig:templates}, at the right corresponds to the origami
manifold which is not quasitoric.

Since we want to find a quasitoric manifold which is not toric
origami, we need to consider only the cases when the orbit space
is contractible. Thus in the following we suppose $\Gamma$ is a
tree.

%
%
%
%
%
%
%

\section{Weighted simplicial cell spheres}\label{sectWeightedSpheres}

In the previous section we have seen that quasitoric manifolds are
encoded by the orbit spaces (which are simple polytopes) and
characteristic functions (which are colorings of facets by
elements of $\Zo^n/\pm$). It will be easier, however, to work with
the dual objects, which we call weighted simplicial spheres. To
some extent this approach is equivalent to multi-fans, used to
study origami manifolds in \cite{ma-pa}, but it is more suitable
for our geometrical considerations.

Recall that a \emph{simplicial poset} or \emph{simplicial cell
complex} \cite{bu-pa04} is a finite partially ordered set $S$ such
that:

(1) There is a unique minimal element $\emptyset\in S$,

(2) For each $I\in S$ the interval subset $[\emptyset, I]\eqd
\{J\in S\mid J\leqslant I\}$ is isomorphic to the poset of faces
of $(k-1)$-dimensional simplex (i.e. Boolean lattice of rank $k$)
for some $k\geqslant 0$. In this case the element $I$ is said to
have rank $k$ and dimension $k-1$.

The elements of $S$ are called simplices and elements of rank $1$
are called vertices. The set of vertices of $S$ is denoted
$\ver(S)$.

A simplicial poset is called pure, if all maximal simplices have
the same dimension. A simplicial poset $S$ is called a simplicial
complex, if for any subset of vertices $\sigma\subseteq\ver(S)$,
there exists at most one simplex whose vertex set is $\sigma$.

\begin{con} It is convenient to visualize simplicial posets using their
geometrical realizations. To define the geometrical realization we
assign the geometrical simplex $\Delta_I$ of dimension $\rk(I)-1$
to each $I\in S$ and attach them together according to the order
relation in $S$. More formally, the geometric realization of $S$
is the topological space
\[
|S|\eqd \bigsqcup\limits_{I\in S} (I,\Delta_I) / \sim,
\]
where $(I_1,x_1)\sim(I_2,x_2)$ if $I_1<I_2$ and $x_1=x_2\in
\Delta_{I_1}\subset \Delta_{I_2}$. See details in \cite{lu-pa}
or~\cite{bu-pa14}.
\end{con}

A simplicial poset $S$ is called a \emph{simplicial cell sphere}
if $|S|$ is homeomorphic to a sphere. $S$ is called a PL-sphere if
it is PL-homeomorphic to the boundary of a simplex. In dimension
2, which is the most important case for us, these two notions are
equivalent. Simplicial complex, whose geometric realization is
homeomorphic to a sphere, is called a \emph{simplicial sphere}.
Thus simplicial sphere is a simplicial cell sphere which is also a
simplicial complex.

\begin{con}\label{conLinkStarSumPosets}
We want to define a connected sum of two simplicial cell spheres
along their vertices. The topological meaning of this operation is
clear: cut the small open neighborhoods of vertices and attach the
boundaries if possible. However, an attempt to define the
connected sum combinatorially for the most general simplicial
posets leads to some technical problems. To keep things
manageable, we exclude certain degenerate situations.

For every $I<J$ in $S$ there is a complementary simplex
$J\smallsetminus I\in S$, since the interval $[\varnothing,J]$ is
identified with the Boolean lattice. In other words,
$J\smallsetminus I$ is the face of $J$ complementary to the face
$I$. Define a \emph{link} of a simplex $I\in S$ as a partially
ordered set $\link_SI=\{J\smallsetminus I\mid J\in S, J\geqslant
I\}$ with the order relation induced from $S$.
Define an \emph{open star} of a simplex $I\in S$ as a subset
$\staro_SI\eqd\{J\in S\mid J\geqslant I\}$. There is a natural
surjective map of sets $D_I\colon\staro_SI\to\link_SI$ sending $J$
to $J\smallsetminus I$. We call a simplex $I$ \emph{admissible} if
$D_I$ is injective.

\begin{figure}[h]
\begin{center}
    \begin{tikzpicture}[scale=0.9]
        \draw (0,2)..controls (0,0.5) and
        (0.5,0)..(2,0)..controls (2,1.5) and
        (1.5,2)..(0,2)--cycle;
        \fill (2,0) circle(2pt) node[anchor=west] {$a$};
        \fill (0,2) circle(2pt) node[anchor=east] {$b$};
    \end{tikzpicture}
\end{center}
\caption{Example of non-admissibility.} \label{fig:nonadmiss}
\end{figure}
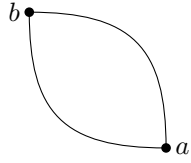

Note that in a simplicial complex every simplex is admissible. An
example of non-admissible simplex is shown on
Fig.~\ref{fig:nonadmiss}. There are two simplices containing the
vertex $a$, and the complement of $a$ in both of them is the same
vertex $b$. Thus $a$ is a non-admissible vertex.

Let us define the connected sum of two simplicial posets $S_1$ and
$S_2$ along admissible vertices. Let $i_1\in S_1$ and $i_2\in S_2$
be admissible vertices, and suppose there exists an isomorphism of
posets $\xi\colon \link_{S_1}i_1\to \link_{S_2}i_2$ (thus an
isomorphism of open stars, by admissibility). Consider a poset
\begin{equation}\label{eqConSumOfPosets}
S_1\csum{i_1}{i_2}S_2\eqd(S_1\smallsetminus \staro_{S_1}i_1)\sqcup
(S_2\smallsetminus \staro_{S_2}i_2)/\sim,
\end{equation}
where $I_1\in\link_{S_1}i_1\subset S_1$ is identified with
$I_2\in\link_{S_2}i_2\subset S_2$ whenever $I_2= \xi(I_1)$. The
order relation on $S_1\csum{i_1}{i_2}S_2$ is induced from $S_1$
and $S_2$ in a natural way. It can be easily checked that the
connected sum $S_1\csum{i_1}{i_2}S_2$ is again a simplicial poset.

If $S_1, S_2$ are simplicial spheres, then so is
$S_1\csum{i_1}{i_2}S_2$. This statement would fail if we do not
impose the admissibility condition.
\end{con}

\begin{rem}
A connected sum of two simplicial complexes may not be a
simplicial complex (Fig.~\ref{fig:consumtriangles}). This is the
main reason why we consider a class of simplicial posets instead
of simplicial complexes.

\begin{figure}[h]
\begin{center}
    \begin{tikzpicture}[scale=1]
        \draw (0,0)--(1.5,1)--(0,2)--cycle;
        \draw (2.5,1)--(4,0)--(4,2)--cycle;
        \draw (5.5,0)..controls (5,0.7) and (5,1.3)..(5.5,2)..controls (6,1.3) and (6,0.7)..(5.5,0)--cycle;
        \node at (2,1) {$\hash$};
        \node at (4.5,1) {$=$};
        \fill (0,0) circle(2pt);
        \fill (0,2) circle(2pt);
        \fill[red] (1.5,1) circle(2pt);
        \fill[red] (2.5,1) circle(2pt);
        \fill (4,0) circle(2pt);
        \fill (4,2) circle(2pt);
        \fill (5.5,0) circle(2pt);
        \fill (5.5,2) circle(2pt);
    \end{tikzpicture}
\end{center}
\caption{The class of simplicial complexes is not closed under
taking connected sums.} \label{fig:consumtriangles}
\end{figure}
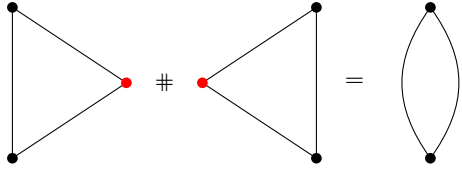

\end{rem}

\begin{defin}
Let $S$ be a pure simplicial poset of dimension $n-1$. A map
$\Lambda\colon\ver(S)\to\Zo^n/\pm$ is called a characteristic
function if, for every simplex $I\in S$ with vertices
$i_1,\ldots,i_n$, the vectors $\Lambda(i_1),\ldots,\Lambda(i_n)$
span $\Zo^n$. The pair $(S,\Lambda)$ is called a weighted
simplicial poset.
\end{defin}


\begin{defin}
Let $(S_1,\Lambda_1)$ and $(S_2,\Lambda_2)$ be weighted simplicial
posets. Let $i_1,i_2$ be admissible vertices of $S_1,S_2$ such
that there exists an isomorphism
$\xi\colon\link_{S_1}i_1\to\link_{S_2}i_2$ preserving
characteristic functions:
$(\Lambda_2\circ\xi)|_{\link_{S_1}i_1}=\Lambda_1|_{\link_{S_1}i_1}$.
Then $\Lambda_1,\Lambda_2$ induce the characteristic function
$\Lambda$ on the connected sum $S_1\csum{i_1}{i_2}S_2$. The
weighted simplicial poset $(S_1\csum{i_1}{i_2}S_2,\Lambda)$ is
called a weighted connected sum of $(S_1,\Lambda_1)$ and
$(S_2,\Lambda_2)$.
\end{defin}

\begin{con}
Let $(P,\Lambda)$ be a characteristic pair (see
section~\ref{sectTopology}). Let $K_P=\partial P^*$ be the dual
simplicial sphere to a simple polytope $P$. Since there is a
natural correspondence $\ver(K_P)=\F(P)$ we get the characteristic
function $\Lambda\colon \ver(K_P)\to \Zo^n/\pm$. This defines a
weighted sphere $(K_P,\Lambda)$. In particular, any Delzant
polytope $P$ defines a weighted sphere $(K_P,\nu)$, where $\nu(F)$
is the normal vector to $F\in \F(P)=\ver(K_P)$ modulo sign
(construction~\ref{conModelSymplectic}).
\end{con}

\begin{con}\label{conAugPosetOfOrigami}
Let $\Ot=(\Gamma,\PsiV,\PsiE)$ be an origami template and
$M_{\Ot}$ be the corresponding toric origami manifold. Suppose
that $\Gamma$ is a tree. The orbit space $Q=M_{\Ot}/T^n$ is
homeomorphic to an $n$-dimensional disc. The face structure of $Q$
defines a poset $S_Q$, whose elements are faces of $Q$ ordered by
reversed inclusion (it is easy to show that such poset is
simplicial). In particular, $\ver(S_Q)=\F(Q)$. Normal vectors to
facets of $Q$ (construction~\ref{conOrbitOfOrigami}) determine the
characteristic function $\nu\colon\F(Q)\to \Zo^n/\pm$,
$\nu([F])=\nu(F)$. Thus there is a weighted simplicial poset
$(S_Q,\nu)$ associated with a toric origami manifold $M_\Ot$.
\end{con}

Our next goal is to describe the weighted simplicial cell sphere
$(S_Q,\nu)$ of a toric origami manifold as a connected sum of
elementary pieces, corresponding to Delzant polytopes of the
origami template.

\begin{con}\label{conAugPosetOfOrigamiAsCSum}
If $\Gamma$ is a tree, then the simplicial poset $S_Q$ is the
connected sum of simplicial spheres $K_{\PsiV(v)}$ along vertices,
corresponding to fold facets:
\begin{equation}\label{eqOrigamiPosetAsCsum}
S_Q\cong\bighash_{\Gamma}K_{\PsiV(v)}.
\end{equation}
Let us introduce some notation to make this precise. Let $e$ be an
edge of $\Gamma$, and $v$ be its endpoint. Let $i_{v,e}$ be the
vertex of $K_{\PsiV(v)}$ corresponding to the facet
$\PsiE(e)\subset\PsiV(v)$. Then \eqref{eqOrigamiPosetAsCsum}
denotes the connected sum of all simplicial spheres $K_{\PsiV(v)}$
along vertices $i_{v,e}$, $i_{u,e}$ for all edges $e=\{v,u\}$ of
graph $\Gamma$. This simultaneous connected sum is well defined.
Indeed, if $e_1\neq e_2\in E$ are two edges emanating from $v\in
V$, then the vertices $i_{v,e_1}$ and $i_{v,e_2}$ are not adjacent
in $K_{\PsiV(v)}$ by pt.2 of
Definition~\ref{definOrigamiTemplate}. Therefore, open stars
$\staro_{K_{\PsiV(v)}}i_{v,e_1}$ and
$\staro_{K_{\PsiV(v)}}i_{v,e_2}$, which we remove in
\eqref{eqConSumOfPosets}, do not intersect. Also note that all
vertices $i_{v,e}$ are admissible, since the spheres
$K_{\PsiV(v)}$ are simplicial complexes.

Each sphere $K_{\PsiV(v)}$ comes equipped with a characteristic
function $\nu_v\colon \ver(K_{\PsiV(v)})\to\Zo^n/\pm$, since
$\PsiV(v)$ is Delzant. By pt.1 of
Definition~\ref{definOrigamiTemplate} these characteristic
functions agree on the links which we identify. Therefore we have
an isomorphism of weighted spheres
\begin{equation}\label{eqOrPosetAugIsCSum}
(S_Q,\nu)\cong\bighash_{\Gamma}(K_{\PsiV(v)},\nu_v).
\end{equation}
\end{con}

%
%
%
%
%
%
%

\section{Proof of Theorem
\ref{thmQtorNotOrigami}}\label{sectProof}

Suppose that a quasitoric manifold $M_{(P,\Lambda)}$ is
equivariantly homeomorphic to the origami manifold
$M_{(\Gamma,\PsiV,\PsiE)}$. As was mentioned earlier, in this
situation $\Gamma$ is a tree.

First, the orbit spaces should be isomorphic as manifolds with
corners: $P\cong Q=M_{\Ot}/T^n$. Second, $M_{(P,\Lambda)}\congeq
M_{\Ot}$ implies that stabilizers of the torus actions coincide
for the corresponding faces of orbit spaces. Thus characteristic
functions on $P$ and $Q$ taking values in $\Zo^n/\pm$ are the
same. Hence, the weighted simplicial cell spheres $(K_P,\Lambda)$
and $(S_Q,\nu) \cong\bighash_{\Gamma}(K_{\PsiV(v)},\nu)$ are
isomorphic.

So far to prove Theorem~\ref{thmQtorNotOrigami} for $n=3$ it is
sufficient to prove the following statement.

\begin{prop}\label{propThmReform}
There exists a 3-dimensional simple polytope $P$ and a
characteristic function $\Lambda\colon\F(P)\to\Zo^3/\pm$ such that
the dual weighted sphere $(K_P,\Lambda)$ cannot be represented as
a connected sum, along a tree, of weighted spheres dual to Delzant
polytopes.
\end{prop}

The proof of this proposition takes most part of this section. We
proceed by steps. At first notice that any simplicial 2-sphere is
dual to some simple 3-polytope by Steinitz's theorem (see e.g.
\cite{zieg}). So it is sufficient to prove that there exists a
weighted 2-dimensional simplicial sphere $K$ which cannot be
represented as a connected sum of weighted spheres dual to Delzant
polytopes.

\begin{figure}[h]
\begin{center}
\includegraphics[scale=0.14]{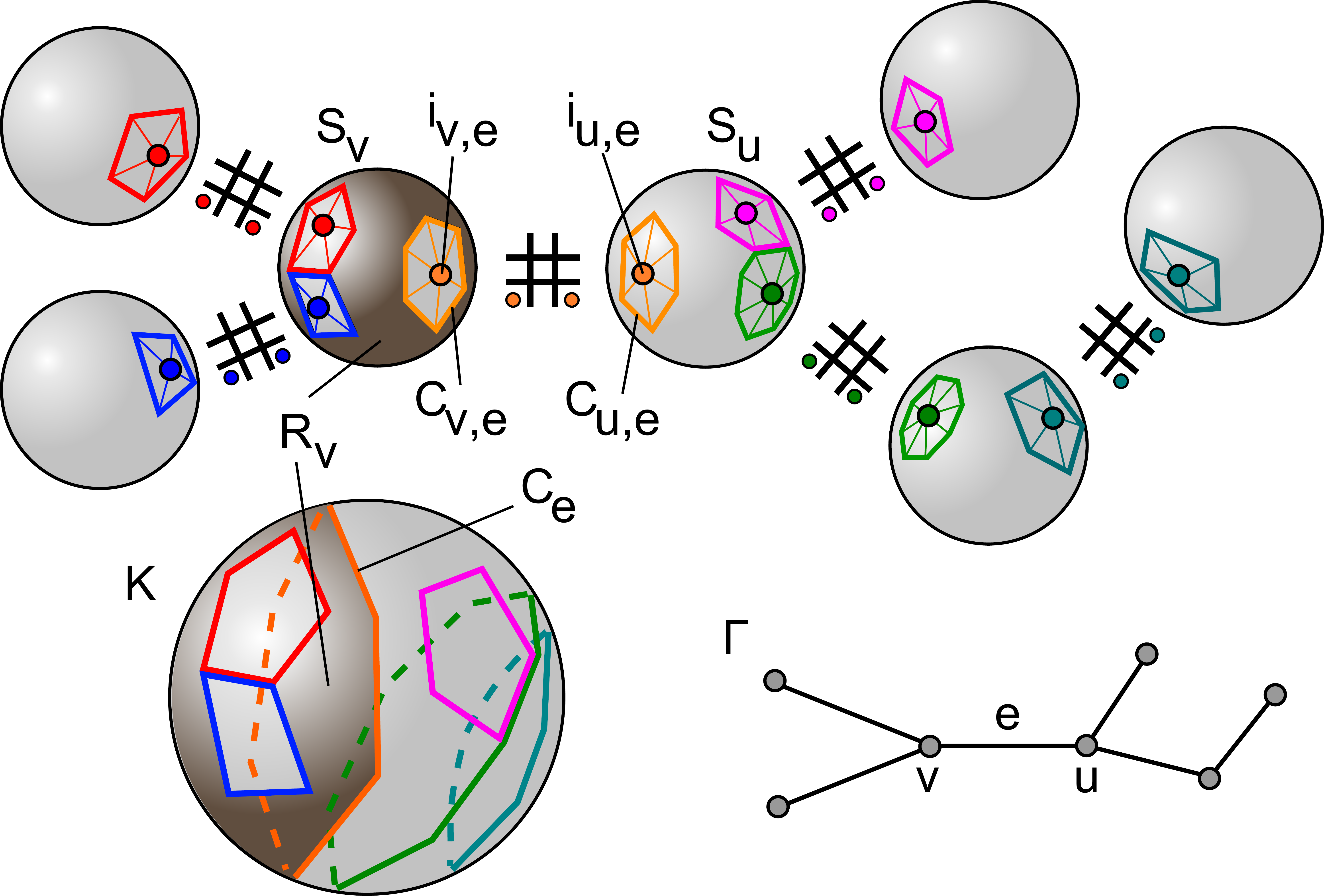}
\end{center}
\caption{Connected sum of spheres along a tree}
\label{pictConSumGraph}
\end{figure}

\begin{con}\label{conSlicingRegions}
We introduce some notation in addition to that of
construction~\ref{conAugPosetOfOrigamiAsCSum}, see
Fig.~\ref{pictConSumGraph}. As before, let $\Gamma=(V,E)$ be a
tree. Suppose that a simplicial cell $(n-1)$-sphere $S_v$ is
associated with each node $v\in V$, and for each edge $e\in E$
with an endpoint $v\in V$ there is an admissible vertex
$i_{v,e}\in S_v$ subject to the following conditions: (1)
$\link_{S_v}i_{v,e}$ is isomorphic to $\link_{S_u}i_{u,e}$ for any
edge $e$ with endpoints $v,u$; (2) Vertices $i_{v,e_1}, i_{v,e_2}$
are different and not adjacent in $S_v$ for any two edges $e_1\neq
e_2$ emanating from $v$. Then we can form a connected sum along
$\Gamma$ as in construction~\ref{conAugPosetOfOrigamiAsCSum}:
$K=\bighash_{\Gamma}S_v$. For each $v\in V$ consider the
simplicial subposet
\begin{equation}\label{eqRegionDefin}
R_v=S_v\bsm\bigsqcup_{e\in E, v\in e} \staro_{S_v}i_{v,e}.
\end{equation}
This subposet will be called a region. Denote $\link_{S_v}i_{v,e}$
by $C_{v,e}$. By construction, $C_{v,e}$ is attached to $C_{u,e}$
if $e=\{v,u\}$. The resulting $(n-2)$-dimensional simplicial
subposet of $K$ is denoted by $C_e$. Since $i_{v,e}$ is
admissible, the subposet $C_e\cong C_{v,e}=\link_{S_v}i_{v,e}$ is
a homological $(n-2)$-sphere as follows from a standard argument
in combinatorial topology (see, e.g.,
\cite[Prop.2.2.14]{bu-pa14}).

We get a collection of $(n-2)$-dimensional cycles $C_e, e\in E$,
dividing the $(n-1)$-sphere $K$ into regions $R_v, v\in V$. If
$e=\{v,u\}$, then $R_v$ and $R_u$ share a common border $C_e$.
Note that cycles $C_e$ are mutually ordered, meaning that each
$C_e$ lies at one side of any other cycle. Though the cycles may
have common points (as schematically shown on
Fig.~\ref{pictConSumGraph}) and even coincide (in this case the
region between them coincides with both of them).

On the other hand, any collection of mutually ordered
$(n-2)$-dimensional spherical cycles in $K$ determines the
representation of $K$ as a connected sum of smaller simplicial
cell spheres. A representation $K=\bighash_{\Gamma}S_v$ will be
called a \emph{slicing}.

Define the \emph{width} of a slicing $\Theta$ to be the maximal
number of vertices in its regions:
\begin{equation}\label{eqWidthOfSlicing}
\wid(\Theta)\eqd\max\{|\ver(R_v)|\mid v\in V\}.
\end{equation}

Define the \emph{fatness} of a sphere $K$ as the minimal width of
all its possible slicings:
\begin{equation}\label{eqFatnessDefin}
\ft(K)\eqd\min\{\wid(\Theta)\mid \Theta \mbox{ is a slicing of }
K\}.
\end{equation}
\end{con}

The essential idea in the proof of Proposition~\ref{propThmReform}
is the following.

\begin{lemma}\label{lemBigFatSmalCol}
Let $K$ be an $(n-1)$-dimensional simplicial cell sphere and
$\Lambda\colon\ver(K)\to\Zo^n/\pm$ a characteristic function. Let
$r$ denote the number of different values of this characteristic
function, $r=|\Lambda(\ver(K))|$. Suppose that $\ft(K)>2r$. Then
$(K,\Lambda)$ cannot be represented as a connected sum, along a
tree, of simplicial spheres dual to Delzant polytopes.
\end{lemma}

\begin{proof}
Assume the converse. Then $(K,\Lambda)\cong
\bighash_{\Gamma}(K_{\PsiV(v)},\nu_v)$, where $\PsiV(v)$ are
Delzant polytopes. Forgetting characteristic functions gives a
slicing $\Theta$ of $K$. The width of every slicing of $K$ is
greater than $2r$ by the definition of fatness. In particular,
$\wid(\Theta)>2r$. Thus there exists a node $v$ of $\Gamma$ such
that $|\ver(R_v)|>2r$.

The region $R_v$ is a subcomplex of $K_{\PsiV(v)}$. The
restriction of $\Lambda$ to the subset $\ver(R_v)$ coincides with
the restriction of $\nu \colon \ver(K_{\PsiV(v)})\to \Zo^n/\pm$ to
$\ver(R_v)$. Recall, that $\tilde{\nu}(F)\in \Zo^n$ is the outward
normal vector to the facet $F\in\F(\PsiV(v))=\ver(K_{\PsiV(v)})$,
and $\nu(F)\in \Zo^n/\pm$ is its class modulo sign. The outward
normal vectors to facets of a convex polytope are mutually
distinct, thus $|\tilde{\nu}(\ver{R_v})|=|\ver(R_v)|$ and,
therefore, $|\nu(\ver(R_v))|\geqslant |\ver(R_v)|/2$. Thus
$|\Lambda(\ver(R_v))|=|\nu(\ver(R_v))|>r$, --- the contradiction,
since $r$ is the total number of values of~$\Lambda$.
\end{proof}

So far we may find counterexamples to origami realizability among
polytopes, which are $\Zo^n$-colored with a small number of
colors, but whose dual simplicial spheres have large fatness. Of
course such examples do not appear when $n=2$ --- this would
contradict Theorem~\ref{thm4dim}. A simplicial 1-sphere is a cycle
graph $\mathcal{C}_k$. By considering diagonal triangulations of a
$k$-gon, one can easily check that $\mathcal{C}_k$ can be
represented as a connected sum of several cycle graphs of the form
$\mathcal{C}_4$ or $\mathcal{C}_5$, giving the slicing of width
$3$. Hence fatness of any $1$-dimensional simplicial sphere is at
most $3$, while any characteristic function takes at least $2$
values, so the conditions of Lemma~\ref{lemBigFatSmalCol} are not
satisfied if $n=2$.

The existence of 2-spheres satisfying conditions of
Lemma~\ref{lemBigFatSmalCol} is thus our next and primary goal. At
first, we prove that any $2$-sphere admits a characteristic
function with few values.

\begin{lemma}\label{lemFourColor}
Any simplicial 2-sphere $K$ admits a characteristic function
$\Lambda\colon\ver(K)\to\Zo^3/\pm$ such that
$|\Lambda(\ver(K))|\leqslant 4$.
\end{lemma}

\begin{proof}
Four color theorem states that there exists a proper
vertex-coloring: $\ver(K)\to
\{\alpha_1,\alpha_2,\alpha_3,\alpha_4\}$. Now replace colors by
integral vectors $\alpha_1\mapsto (1,0,0)$, $\alpha_2\mapsto
(0,1,0)$, $\alpha_3\mapsto (0,0,1)$, $\alpha_4\mapsto (1,1,1)$.
Any three of these four vectors span the lattice. Therefore the
map $\Lambda\colon \ver(K)\to \Zo^3/\pm$ thus obtained is a
characteristic function. It takes at most $4$ values.
\end{proof}

\begin{rem}
This is a standard trick in toric topology. Classically, it is
applied to prove that any simple 3-polytope admits a quasitoric
manifold~\cite{da-ja}.
\end{rem}

Though for our purpose we just need $2$-spheres $K$ with
$\ft(K)=9$, it seems intuitively clear that in dimension $2$ and
higher there exist spheres of arbitrarily large fatness. But it is
not a priori clear how to describe such spheres explicitly in
combinatorial terms. We present one possible approach below, but
some steps of our construction do not generalize to dimensions
greater than $2$.

\begin{prop}\label{propExistFatPoly}
For any $N>0$ there exists a simplicial 2-sphere $K$ such that
$\ft(K)>N$.
\end{prop}

\begin{figure}[h]
\begin{center}
\includegraphics[scale=0.15]{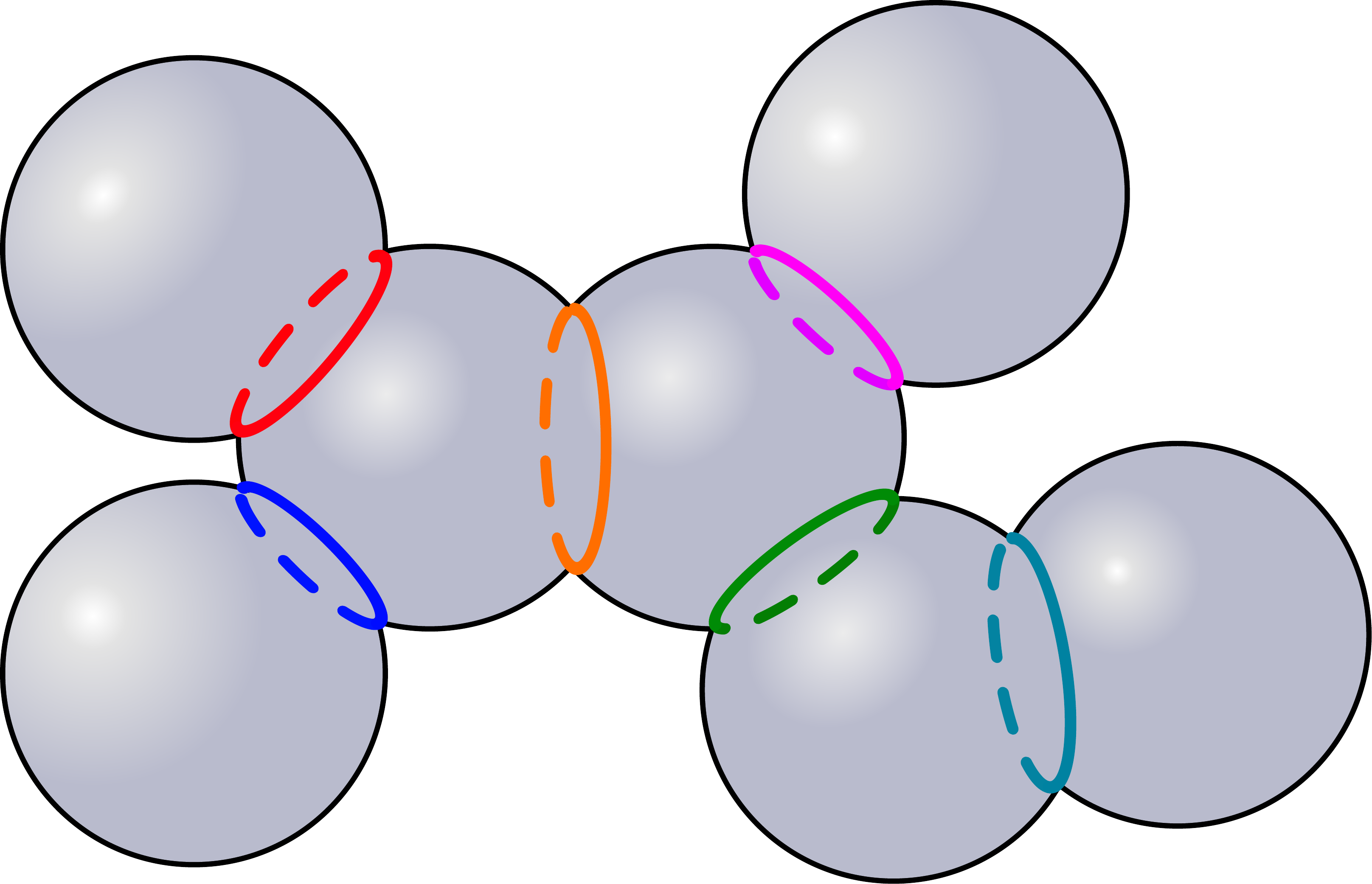}
\end{center}
\caption{Metric features of a ``thin'' simplicial sphere}
\label{pictCactus}
\end{figure}

\begin{proof}
The underlying idea is the following. Suppose that a $2$-sphere
$K$ is ``thin'' i.e. $\ft(K)\ll |\ver(K)|$. Then there exists a
slicing $\Theta$ of $K$ into pieces with small numbers of
vertices. In particular, the discrete length of any cycle $C_e$ in
a slicing $\Theta$ should be small. Then the sphere $K$ is
``tightened'', like the one shown on Fig.~\ref{pictCactus}. It has
the feature that small cycles can bound large areas. To measure
this property, we introduce a natural metric on $|K|$ in which all
edges have length $1$, and then compare the metric space $|K|$
with a ($2$-dimensional) round sphere $\Ss$ of a constant radius.
If there is a bijection $|K| \leftrightarrows \Ss$ with close
Lipschitz constants, then $|K|$ is not thin. The reason why $\Ss$
suits well for this consideration is that small curves on $\Ss$
cannot bound large areas, which follows from the isoperimetric
inequality.

Quite similar considerations and ideas are used in the theory of
planar separators. Some results of this theory can be used to
prove Proposition~\ref{propExistFatPoly} directly. If $G$ is a
planar graph with $k$ vertices, it is known that there exists a
set of $O(\sqrt{k})$ vertices which separates $G$ into two parts
of roughly equal size (such separating sets are called planar
separators). It is also known that the asymptotic $O(\sqrt{k})$ is
the best possible for planar separators~\cite{dj}. If all
$2$-spheres were ``thin'', then every planar graph would have a
separator of size, bounded by some constant, which contradicts the
aforementioned asymptotic.

Anyway the deduction of Proposition \ref{propExistFatPoly} from
the known theory requires some additional work, so we give an
independent proof. Now that we described the intuitive idea beyond
our approach let us get to technical work.

\begin{con}\label{conMetMeasOnK}
Let $K$ be a $2$-dimensional simplicial complex. Define a
(piecewise Riemannian) metric $g$ and measure $\mu$ on $|K|$ in
such a way that each triangle $|I|\subset |K|$ becomes an
equilateral Euclidian triangle with the standard metric and edge
length $1$. Thus the area of each triangle is $\sqrt{3}/4$.

Let $L(\gamma)$ denote the length of a piecewise smooth curve
$\gamma$ in $|K|$. If $C\subset K$ is a closed $1$-dimensional
cycle (simplicial subcomplex), then, obviously,
\begin{equation}\label{eqDiscrLength}
L(|C|) = |\ver(C)|.
\end{equation}
A cycle $C$ divides $K$ into two subcomplexes $K_+$ and $K_-$,
each homeomorphic to a closed $2$-disc (we suppose $C\subset
K_+,K_-$). Let us estimate the number of vertices in $K_-$ in
terms of its area ($K_+$ is similar). Let
$\mathcal{V}_-,\mathcal{E}_-,\mathcal{T}_-$ denote the number of
vertices, edges and triangles in $K_-$. By the definition of
measure, $\mathcal{T}_-=\frac{4}{\sqrt{3}}\mu(|K_-|)$. We have
$\mathcal{V}_--\mathcal{E}_-+\mathcal{T}_-=1$ (Euler
characteristic of $K_-$) and $\mathcal{E}_-<3\mathcal{T}_-$ (by
counting pairs $e\subset t$, where $e$ is an edge and $t$ is a
triangle). Therefore,
\begin{equation}\label{eqDiscrArea}
\mathcal{V}_-\leqslant \frac{8}{\sqrt{3}}\mu(|K_-|).
\end{equation}
\end{con}

Let $\Ss_R$ be a $2$-dimensional round sphere of radius $R$, with
the standard metric $g_s$ and measure $\mu_s$. A piecewise smooth
closed curve $\gamma\subset \Ss_R$ without self-intersections
divides $\Ss_R$ into two regions $A_+$, $A_-$. The isoperimetric
inequality on a sphere (see e.g.~\cite[Ch.4]{os}) has the form

\begin{equation}
R^2L_s(\gamma)^2\geqslant \mu_s(A_+)\mu_s(A_-),
\end{equation}
where $L_s(\gamma)$ is the length of $\gamma$. Since
$\mu_s(\Ss_R)=4\pi R^2$ we may assume that $\mu_s(A_+)\geqslant
2\pi R^2$ (otherwise consider $A_-$ instead), thus

\begin{equation}\label{eqIsopMain}
\mu_s(A_-)\leqslant \frac{L_s(\gamma)^2}{2\pi}.
\end{equation}
Notice that this inequality does not depend on the sphere radius.

Let $K$ be a $2$-dimensional simplicial sphere and
$R,c_1,c_2,c_3,c_4$ be positive real numbers. Suppose there exists
a bijective piecewise smooth map $f\colon|K|\to \Ss_R$ such that

\begin{gather}
c_1 L(\gamma)\leqslant L_s(f(\gamma))\leqslant c_2L(\gamma),
\label{eqLipschLength}\\
c_3\mu(\Omega)\leqslant \mu_s(f(\Omega))\leqslant
c_4\mu(\Omega),\label{eqLipschArea}
\end{gather}
for each piecewise smooth curve $\gamma\subset|K|$ and measurable
set $\Omega \subset |K|$. Numbers $c_1,c_2,c_3,c_4$ will be called
Lipschitz constants of the map $f$.

\begin{lemma}\label{lemSmallPartOfCycleK}
Suppose there exists a mapping $f$ of 2-sphere $K$ into a round
sphere $\Ss_R$ with Lipschitz constants $c_1,c_2,c_3,c_4$. Let $C$
be a cycle of $K$ dividing it into two closed regions $K_+$ and
$K_-$. If $C$ is contains at most $N$ vertices, then either $K_+$
or $K_-$ contains at most $\frac{4N^2c_2^2}{\sqrt{3}\pi c_3}$
vertices.
\end{lemma}

\begin{proof}
Among two regions $f(|K_-|),f(|K_+|)\subset \Ss_R$ let $f(|K_-|)$
be the one with the smaller area. Combining \eqref{eqDiscrLength},
\eqref{eqDiscrArea}, \eqref{eqIsopMain}, \eqref{eqLipschLength},
and \eqref{eqLipschArea}, we get
\begin{equation}
\mathcal{V}_-\leqslant \frac{8}{\sqrt{3}}\mu(|K_-|)\leqslant
\frac{8\mu_s(f(|K_-|))}{\sqrt{3}c_3}\leqslant
\frac{8L_s(f(|C|))^2}{2\sqrt{3}\pi c_3}\leqslant \frac{4
N^2c_2^2}{\sqrt{3} \pi c_3},
\end{equation}
which was to be proved.
\end{proof}

Suppose $\ft(K)\leqslant N$. Then by definition there exists a
slicing $K=\bighash_{\Gamma}S_v$, encoded by a tree $\Gamma$, such
that each region $R_v$ has at most $N$ vertices (see
construction~\ref{conSlicingRegions}). Let us show that the degree
of each node $v$ of $\Gamma$ is bounded from above.

\begin{lemma}\label{lemDegreeNodeBound}
If $\Theta$ is a slicing $K=\bighash_{\Gamma}S_v$ and
$\wid(\Theta)\leqslant N$, then $\deg v\leqslant 2(N-2)$ for any
node $v$ of $\Gamma$.
\end{lemma}

\begin{proof}
Denote $\deg v$ by $d$. By construction, the region $R_v$ is
obtained from a sphere $S_v$ by removing $d$ open stars which
correspond to the edges of $\Gamma$ emanating from $v$. The
complex $R_v$ itself can be considered as a plane graph. Denote
the numbers of its vertices, edges and faces by
$\mathcal{V},\mathcal{E},\mathcal{R}$ respectively. By the
definition of the width, we have $\mathcal{V}\leqslant N$. We also
have $\mathcal{V}-\mathcal{E}+\mathcal{R}=2$, and
$2\mathcal{E}\geqslant 3\mathcal{R}$ (each region has at least $3$
edges). Thus, $\mathcal{V}\geqslant 2+\frac12\mathcal{R}$. Notice
that each removed open star represents a face of graph $R_v$,
therefore, $d\leqslant \mathcal{R}\leqslant
2(\mathcal{V}-2)\leqslant 2(N-2)$.
\end{proof}

\begin{lemma}\label{lemManyVertImplyFat}
Let $K$ be a $2$-dimensional simplicial sphere endowed with the
map $f$ to the round sphere, satisfying Lipschitz bounds
\eqref{eqLipschLength} and \eqref{eqLipschArea}. For a natural
number $N$ set $A=\frac{4N^2c_2^2}{\sqrt{3}\pi c_3}$ and $B=
2(N-2)$. If $|\ver(K)|>\max(AB+N,2A)$, then $\ft(K)>N$.
\end{lemma}

\begin{proof}
Assume the contrary: $\ft(K)\leqslant N$. Then there is a slicing
$K=\bighash_{\Gamma}S_v$ in which every region $R_v$ has at most
$N$ vertices. Consequently, any cycle $C_e, e\in E$ has at most
$N$ vertices. By Lemma~\ref{lemSmallPartOfCycleK}, the cycle $C_e$
divides $K$ into two parts, one of which has $\leqslant A$
vertices. Since $|\ver(K)|>2A$, the other part has $>A$ vertices.
Assign a direction to each edge $e$ of $\Gamma$ in such a way that
$e$ points from the larger component of $K\smallsetminus C_e$ to
the smaller, where the ``size'' means the number of vertices.

$\Gamma$ is a tree, therefore there exists a source, i.e. a node
$u$ from which all adjacent edges emanate. Speaking informally,
this node represents a ``big sized bubble'', meaning that the part
of a sphere, lying across each border has a small size. Let $d$
denote the degree of the chosen node $u$. Denote by
$\Gamma_1,\ldots,\Gamma_d$ the connected components of the graph
$\Gamma\smallsetminus u$. By Lemma~\ref{lemDegreeNodeBound} we
have $d\leqslant B$. By the construction of the directions of
edges, $\left|\ver(\bigsqcup_{\Gamma_i}R_v)\right|\leqslant A$ for
each $\Gamma_i$. Thus $|\ver(K)|\leqslant|\ver(R_u)|+\sum_{i=1}^d
\left|\ver(\bigsqcup_{\Gamma_i}R_v)\right| \leqslant N+AB$ --- the
contradiction.
\end{proof}

\begin{lemma}\label{lemExistBigAndLips}
For any $N>0$ there exists a $2$-dimensional simplicial sphere $K$
such that:
\begin{enumerate}
\item There exists a piecewise smooth map $f\colon|K|\to\Ss_R$ satisfying Lipschitz bounds
\eqref{eqLipschLength} and \eqref{eqLipschArea} for some constants
$c_1,c_2,c_3,c_4,R>0$
\item $|\ver(K)|>\max(AB+N,2A)$, where $A$ and $B$ are defined in
Lemma~\ref{lemManyVertImplyFat}.
\end{enumerate}
\end{lemma}

\begin{proof}
Start with the boundary of a regular tetrahedron with edge length
$1$: $L=\partial \Delta^3$. The projection from the center of $L$
to the circumsphere $f\colon L\to \Ss_R$ is obviously Lipschitz
for some constants $c_1,c_2,c_3,c_4>0$. Now subdivide each
triangle of $|L|$ into $q^2$ smaller regular triangles as shown on
Fig.~\ref{fig:subdiv}.

\begin{figure}[h]
\begin{center}
    \begin{tikzpicture}[scale=0.5,fill=black!35]
    \tikzstyle{myarrows}=[line width=1mm,draw=black!35,-triangle 90,postaction={draw, line width=3mm, shorten >=2mm, -}]
    \def\hh{0.866}
    \foreach \i in {0,1,...,4} {
        \pgfmathtruncatemacro{\en}{4-\i}
        \foreach \j in {0,...,\en} {
            \draw (0.5*\i + \j,\hh*\i) -- (0.5*\i + \j + 1,\hh*\i) -- (0.5*\i + \j + 0.5,\hh*\i+\hh) -- cycle;
        }
    }
    \draw (-7,0) -- (-2,0) -- (-4.5,5*\hh) -- cycle;
    \draw [myarrows](-2,2.5*\hh)--(0,2.5*\hh);
    \end{tikzpicture}
\end{center}
\caption{Subdivision of a regular triangle.} \label{fig:subdiv}
\end{figure}
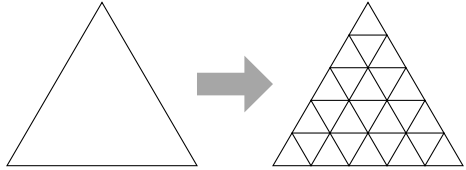

This results in a simplicial complex $L_{(q)}$. As a space with
metric and measure $|L_{(q)}|$ is homothetic to $|L|$ with a
linear scaling factor $q$ (recall that the metric on simplicial
complexes is introduced in such way that each edge has length
$1$). Thus there exists a map $f_{(q)}\colon |L_{(q)}|\to
\Ss_{qR}$ with the same Lipschitz constants as $f$. The number of
vertices $|\ver(L_{(q)})|$ can be made arbitrarily large.

\end{proof}

Lemmas \ref{lemExistBigAndLips} and \ref{lemManyVertImplyFat}
conclude the proof of Proposition~\ref{propExistFatPoly}.
\end{proof}

\begin{rem}
Actually, in the proof of Lemma~\ref{lemExistBigAndLips} we could
have started from any simplicial sphere $L$, take any piecewise
smooth map $f\colon|L|\to \Ss_R$, find Lipschitz constants
$c_2,c_3>0$ (they exist by the standard calculus arguments), and
then apply the same subdivision procedure. We used the boundary of
a regular simplex, because in this case Lipschitz map is
constructed easily and allows for an explicit computation.

We give a concrete example of a quasitoric manifold which is not
toric origami, by performing this computation. The calculations
themselves are elementary thus omitted. It is sufficient to
construct a simplicial sphere for $N=8$. For a projection map from
the boundary of a regular tetrahedron to the circumscribed sphere
we have Lipschitz constants $c_2=3$, $c_3=\frac{1}{3}$. Thus
$\max(AB+N,2A)\approx 15251.14$. Subdivide each triangle in the
boundary of a regular tetrahedron in $q^2$ small triangles where
$q\geqslant 88$. This gives a simplicial sphere $K$ with at least
$15490$ vertices and the same Lipschitz constants as $\partial
\Delta^3$. Thus $\ft(K)>8$. Now take the dual simple polytope $P$
of $K$, consider any proper coloring of facets in four colors and
assign a characteristic function $\Lambda$, as described in
Lemma~\ref{lemFourColor}. This gives a characteristic pair
$(P,\Lambda)$, whose corresponding quasitoric manifold is not
toric origami.

Of course, all our estimations are very rough, and, probably,
there are better ways to construct fat spheres. For sure, there
exist $2$-spheres of fatness $9$ with less than $15490$ vertices.
\end{rem}

\begin{rem}
Note that in dimension $3$ and higher there is no simplicial
subdivision of a regular simplex into smaller regular simplices.
This is one of two places in the proof, where the dimension
restriction is crucial. The second place is the Four color theorem
in Lemma~\ref{lemFourColor}.
\end{rem}

Proposition \ref{propThmReform} proves Theorem
\ref{thmQtorNotOrigami} for $n=3$. Now we need to make the
remaining cases $n>3$.

\begin{prop}\label{propExistHighDim}
There exist quasitoric manifolds of any dimension $2n$, $n>3$,
which are not toric origami.
\end{prop}

\begin{proof}
Let $M_{(P,\Lambda)}$ be any quasitoric manifold, which is not
toric origami. Take the product of $M_{(P,\Lambda)}$ with $S^2$
(the circle $T^1$ acts on $S^2$ by axial rotations). On the level
of orbit spaces, this corresponds to multiplying $P$ with a closed
interval $\mathbb{I}\subset \Ro$. We claim that quasitoric
manifold $M_{(P,\Lambda)}\times S^2$ is not toric origami. If
$M_{(P,\Lambda)}\times S^2$ were a toric origami manifold, then
all its characteristic submanifolds should be toric origami as
well (see construction~\ref{conOrbitOfOrigami}). But
$M_{(P,\Lambda)}$ is one of them. This gives a contradiction. Thus
taking products with $S^2$ produces examples for all $n>3$.
\end{proof}

\begin{rem}
Sphere $S^2$ is the simplest example of a quasitoric manifold. In
the proof of Proposition~\ref{propExistHighDim} we could have used
any other quasitoric manifold instead of $S^2$. If
$M_{(P,\Lambda)}$ and $M_{(P',\Lambda')}$ are quasitoric manifolds
and one of them is not toric origami, then the quasitoric manifold
$M_{(P,\Lambda)}\times M_{(P',\Lambda')} = M_{(P\times
P',\Lambda\oplus\Lambda')}$ is not toric origami as well.
\end{rem}

\begin{rem}
On the other hand, new toric origami manifolds can be produced
from a given one in a similar way as we used for constructing
non-examples. It is easy to observe that if $M$ is a toric origami
manifold and $M'$ is a toric symplectic manifold, then $M\times
M'$ is again a toric origami manifold. We would also like to
mention that projective bundles over toric origami manifolds are
again toric origami manifolds. More precisely, if $M^{2n}$ is
toric origami, and $L_1,\ldots,L_k$ are complex line bundles over
$M$, each having an $S^1$ action on fibers, then the
projectivization
\[
\tilde{M}=P\left(\bigoplus_{j=1}^kL_j\oplus \underline{\Co}\right)
\]
with the induced action of $T^n\times(S^1)^k$ is also a toric
origami manifold.
\end{rem}

%
%
%
%
%
%
%

\section{Discussion and open questions}\label{sectConclusion}

\subsection{Asymptotically most of simplicial $2$-spheres are fat}
We already mentioned a relation of our study to the theory of
planar separators in Section~\ref{sectProof}. We also want to
mention another connection to the theory of random infinite planar
maps. This rapidly developing part of probability theory aims,
among other things, to give a firm foundation for some facts in
statistical physics and quantum gravity. The basic idea of this
study is the following~\cite{gal-lect,gal-main}. Fix a number $k$,
a parameter of the whole construction. For a given $n$ consider
all possible (rooted) plane $k$-angulations with $n$ faces. For
$k=3$, these are roughly the same as simplicial spheres. Every
plane graph has a standard metric, turning it into a metric space.
By letting the number of faces tend to infinity, and renormalizing
the diameter of graphs in a correct way, one considers the limits
of converging sequences of graphs. The limits are taken with
respect to the Gromov--Hausdorff metric defined on the set of
isometry classes of metric spaces.

Since there is only a finite number of such graphs with a fixed
number $n$ of faces, we can take a uniform distribution on this
set of graphs. The uniform distributions on the sets of prelimit
metric spaces give rise to a limiting distribution, which is
viewed as a random compact metric space (of course, here we omit a
lot of technicalities, needed to state everything precisely). The
resulting random metric space is called a Brownian map and
considered as a good $2$-dimensional analogue of the Brownian
motion.

A wonderful thing is that a Brownian map does not actually depend
on the parameter $k$, if $k$ is either $3$ or
even~\cite{gal-main}. It is also known that the Brownian map is
almost surely homeomorphic to a $2$-sphere~\cite{gal-homeo}. This
suggests the following

\begin{claim}
For each $N>0$ almost all simplicial $2$-spheres $K$ have
$\ft(K)>N$. More precisely, if $A_n$ denotes the set of all
simplicial $2$-spheres with $\leqslant n$ triangles, and
$B_{n,N}\subset A_n$ the subset of simplicial spheres having
$\ft(K)>N$, then
\[\lim_{n\to\infty}\frac{|B_{n,N}|}{|A_n|}=1.\]
\end{claim}

The reason is as follows (cf. \cite[Cor.5.3]{gal-lect}). If there
were a lot of ``thin'' simplicial spheres, they all would have
bottlenecks --- small cycles, dividing them into macroscopic
regions. After taking a limit as $n\to\infty$ and rescaling the
metric, these bottlenecks would collapse to points. Thus the
limiting metric space would be non-homeomorphic to a sphere with
non-zero probability.

Therefore, for most of simple combinatorial 3-polytopes $P$ there
exists a characteristic function $\Lambda$ such that
$M_{(P,\Lambda)}$ is not toric origami.

\subsection{Orbit spaces of toric origami manifolds}
We may ask a more intricate question.

\begin{problem}\label{problCombRestr}
Find a simple polytope $P$ such that any quasitoric manifold
$M_{(P,\Lambda)}$ over $P$ is not equivariantly homeomorphic to a
toric origami manifold.
\end{problem}

This question is motivated by the following fact. There exist a
simple 3-polytope $P$ such that any quasitoric manifold
$M_{(P,\Lambda)}$ over $P$ is not equivariantly homeomorphic to a
symplectic toric manifold. Stating shortly: there exist
combinatorial types of simple 3-polytopes which do not admit
Delzant realizations. It was proved in \cite{delaunay} that any
3-dimensional Delzant polytope has at least one triangular or
quadrangular face. Consequently, in particular, a dodecahedron
does not admit a Delzant realization.

An origami template is a generalization of a single Delzant
polytope, thus a realizability of a given combinatorial polytope
by an origami template is a more complicated task.
Problem~\ref{problCombRestr} can be restated in different terms:
are there any combinatorial restrictions on the orbit spaces of
toric origami manifolds?

\subsection{Fat simplicial spheres in higher dimensions}
The examples of non-origami quasitoric manifolds in high
dimensions were constructed from the $3$-dimensional case. On the
other hand, Lemma~\ref{lemBigFatSmalCol} applies for any
dimension. The problem of finding higher-dimensional polytopes
whose dual spheres have large fatness may be of independent
interest.

Actually, even if we find such a fat sphere, to make use of the
developed technique we should also construct a characteristic
function with a small range of values. This constitutes a certain
problem, since characteristic function may not even exist, if
$n\geqslant 4$ (this happens for dual neighborly polytopes,
see~\cite{da-ja}). Nevertheless, there is a big class of
simplicial $(n-1)$-spheres, so called balanced spheres, which
admit a proper vertex-coloring in $n$ colors. Such colorings give
rise to characteristic functions, which have exactly $n$ values,
i.e. minimal possible. Such characteristic functions and the
corresponding quasitoric manifolds were called linear models
in~\cite{da-ja}. Passing to a barycentric subdivision makes every
simplicial sphere into a balanced sphere. We suppose that passing
to a barycentric subdivision does not strongly affect the fatness.
If so, given any fat sphere dual to a simple polytope, one can
pass to its barycentric subdivision, provide it with a linear
model characteristic function, and finally obtain a quasitoric
manifold which is not toric origami.

\subsection{Minimizing the range of characteristic function}
Another problem, which naturally arises from
Lemma~\ref{lemBigFatSmalCol} is to find, for a given polytope $P$,
a characteristic function $\Lambda$ with the minimal possible
range of values $|\Lambda(\F(P))|$, if at least one characteristic
function is known to exist. This minimal number seems to be an
analogue of Buchstaber invariant (see the definition in
\cite{erokh} or \cite{ayz}), as was noted to us recently by
N.\,Erokhovets. It may happen that an interesting theory hides
beyond this subject.

\subsection{Toric varieties}
There exist obstructions to origami realizability, other than
those described in section~\ref{sectProof}. If a weighted
simplicial sphere $K$ can be represented as a connected sum, along
a tree, of simplicial spheres dual to Delzant polytopes, this does
not mean automatically that $K$ corresponds to an origami
template. The reason is that a convex polytope contains more
information than its normal fan (or, in our terminology, dual
weighted simplicial sphere). It can be impossible to assemble an
origami template from a collection of Delzant polytopes, even if
their dual weighted spheres suit together well.

%

Such situations appeared when we tried to answer the following

\begin{problem}
Does there exist a compact smooth toric variety, which is not
equivariantly homeomorphic to a toric origami manifold?
\end{problem}

Any projective toric variety corresponds to a convex polytope.
Thus any smooth projective toric variety is a symplectic toric
manifold, which is a particular case of toric origami. Thus, to
prove the conjecture, one should consider non-projective examples.
Translating the problem into combinatorial language, the task is
to find a complete smooth fan, which is not a normal fan of any
polytope and, moreover, not dual to any origami template. The
simplest non-polytopal fan is the fan corresponding to a famous
non-projective Oda's 3-fold~\cite[p.84]{oda}. So it is natural to
start with a more concrete question:

\begin{problem}
Is Oda's 3-fold a toric origami manifold?
\end{problem}

Even this question happens to be rather non-trivial and cannot be
solved solely by the method developed in this paper.

\subsection{Origami manifolds which are not quasitoric}
In section~\ref{sectTopology} we mentioned that a toric origami
manifold $M_O$ is not quasitoric if its template graph has cycles.
Even if the orbit space of $M_O$ is contractible, the manifold
$M_O$ may not be quasitoric. The simplest example of this kind is
the sphere $S^4$ (example~\ref{exSfour}). The orbit space of $S^4$
is a 2-gon, shown on Fig.~\ref{fig:orbitspace}, which is not a
convex polytope. Excluding situations of these two kinds we may
ask the following question.

\begin{problem}
Let $M_O$ be a simply connected toric origami manifold and suppose
that the dual simplicial sphere of its orbit space is a simplicial
complex. Is the manifold $M_O$ quasitoric?
\end{problem}

In other words, does the orbit space of a simply connected toric
origami manifold admit a convex realization, provided that its
dual simplicial sphere is a simplicial complex?

\end{document}